\documentclass[12pt]{amsart}
\usepackage{amssymb}
\usepackage{graphics}
\usepackage{latexsym}
\usepackage{amsmath}
\usepackage{amssymb}
\usepackage{amscd}
\usepackage{multicol}
\usepackage[cmtip,matrix,arrow]{xy}
\usepackage{hyperref}

\newtheorem{thm}{Theorem}[section]
\newtheorem{prop}[thm]{Proposition}

\newtheorem{lem}[thm]{Lemma}

\theoremstyle{definition}

\newtheorem{ex}[thm]{Example}

\theoremstyle{remark}
\newtheorem{rem}[thm]{Remark}

\newcommand{\RR}{\mathbb R}
\newcommand{\ZZ}{\mathbb Z}

\newcommand{\CC}{\mathbb C}

\begin{document}

\address{Department of Mathematics and Statistics\\ University of Regina\\ Canada}

\email{liviu.mare@gmail.com}

\title{Simply connectedness of spaces of tight frames}

\author{Augustin-Liviu Mare}

%\date{June 30, 2025}

\begin{abstract} Complex tight frames can be canonically viewed as elements of a complex Stiefel manifold. We present a class of spaces of such frames which are simply connected relative to the subspace topology. 
To this class belongs the space of finite unit-norm tight frames.     
\end{abstract}
\maketitle
%\tableofcontents

\section{Statement of the main result}\label{sect1}

Let $0\le k \le n$ be integers. A {\it normalized tight frame} is a matrix $F\in {\rm Mat}_{k\times n}(\CC)$ whose product with its  conjugate transpose  $F^*$ is 
\begin{equation}\label{star}FF^* = {\rm I}_k.\end{equation} For more insight on this notion, see the monograph \cite{Wa}. 
Let $f_1, \ldots, f_n \in \CC^k$ be  column vectors, such that $F=[ \ f_1 \ | \ \ldots \ | \ f_n \ ]$. 
Given $d=(d_1,\ldots, d_n) \in \RR^n$ whose components are all non-negative, we denote by ${\mathcal F}_{n,k}^d$ the collection of all normalized tight frames $F$ with the property that
$$\|f_j\|^2=d_j, \  {\rm for \ all \ } 1 \le j \le n,$$
the norm being the standard one. 
For this space to be non-empty, it is necessary and sufficient that 
\begin{align*}
{}& 0\le d_j \le 1, \ {\rm for \ all \ } 1\le j \le n, \\  
{}& d_1+\cdots + d_n=k,
\end{align*}
see Rem.~\ref{sh} below.
Denote by $\Delta_{n,k}\subset \RR^n$ the polytope described by the equations above. 
It was proved in \cite{NS1} that for any $d \in \Delta_{n,k}$, the space ${\mathcal F}_{n,k}^d$ is path connected 
relative to the topology of subspace of  ${\rm Mat}_{k\times n}(\CC)$. The main result of this note is:

\begin{thm}\label{2mainthm}  If $d=(d_1, \ldots, d_n)\in \Delta_{n,k}$  such that  
\begin{equation} \label{hypothesis} d_{i_1} +\cdots + d_{i_{n-k}}\ge 1 \ {\it 
 for \ any \ pairwise  \  distinct } \ i_1, \ldots, i_{n-k} \in \{1, \ldots, n\}, \end{equation} then ${\mathcal F}_{n,k}^d$
is a simply connected topological subspace of ${\rm Mat}_{k\times n}(\CC)$. 
\end{thm}

\begin{rem}\label{lege}
If assumption 
 (\ref{hypothesis}) is satisfied, then $ 2\le k \le n- 2$, as shown in \cite[Rem.~1.3]{Mare}. Also see Ex.~\ref{exa1} below. 
\end{rem}

\begin{ex}\label{example1}
Given $n$ and $k$ such that $2\le k \le n-2$, there exists $d\in \Delta_{n,k}$ which satisfies 
assumption (\ref{hypothesis}), namely  $$d=\left(\frac{k}{n}, \ldots, \frac{k}{n}\right).$$
The corresponding  frames have been previously investigated by several authors, although in a slightly different presentation,
being called {\it unit-norm} or {\it spherical} tight frames, see  for instance \cite{DS} and \cite{CMS}. 
 \end{ex}

The proof of Thm.~\ref{2mainthm} which is presented below uses a certain relationship between the space ${\mathcal F}_{n,k}^d$ 
and the Grassmannian of $k$-dimensional complex vector subspaces of $\CC^n$. The latter is coming equipped with a canonical symplectic form and a Hamiltonian torus action, the image of the  moment map being just the polytope $\Delta_{n,k}$ above.  
To achieve our goal, we first show  that the pre-image under the moment map of any $d\in \Delta_{n,k}$ which satisfies  assumption (\ref{hypothesis}) is a simply connected subspace of the Grassmannian, see Thm.~\ref{mainthm} below. It  remains to  deduce that ${\mathcal F}_{n,k}^d$ is simply connected as well, which can be done due to its  canonical structure of a ${\rm U}(k)$-principal bundle over the pre-image above. 
Results  from symplectic geometry are fundamental ingredients here. Connections between this branch of mathematics and frame theory were for the first time pointed out in the aforementioned work \cite{NS1} of  Needham and Shonkwiler. Recently, in \cite{CNS},  further progress in studying the topology of ${\mathcal F}_{n,k}^d$
has been made by the same two authors, this time together with Caine, by identifying points $d\in  \Delta_{n,k}$ for which certain higher order homotopy groups of this space vanish. More precisely, this is the content of \cite[Thm.~4.5]{CNS}, and from this, a simply connectedness criterion can immediately be deduced. The latter turns out to be closely related to  
the one given by Thm.~\ref{2mainthm}  above, as explained in Ex.~\ref{exa4}. 
Similarities between \cite[Sect.~4]{CNS} and the present note can also be noticed concerning  the methods used, see also Rem.~\ref{meth} below.

Before starting the proof, we would like to mention that the manifolds in this note are all going to be smooth  and  submanifolds will always be embedded. 
As a first step, we will set the stage by introducing the  Stiefel and the Grassmann manifolds and making some considerations which will turn out to be  useful later on.

\section{Frames via the Stiefel and the Grassmann manifolds}\label{Grass}

We start by endowing $\CC^n$ with the usual Hermitean inner product, that is, $\langle u, v\rangle =uv^*$, for all $u,v \in \CC^n$.  Recall that the Stiefel manifold 
${\rm V}_k(\CC^n)$ consists of all sequences of $k$ orthonormal vectors in
$\CC^n$. 
Any element of this space can be identified with the $k\times n$ matrix $F$ whose rows are exactly the $k$ vectors. 
In this way, ${\rm V}_k(\CC^n)$ is exactly the space of all normalized tight frames, that is, of matrices $F\in {\rm Mat}_{k\times n}(\CC)$
which satisfy condition (\ref{star}). 

To any $k$-tuple of orthonormal vectors one attaches their span, which is a $k$-dimensional complex vector subspace of $\CC^n$,
and obtains in this way a map from ${\rm V}_k(\CC^n)$ to the Grassmannian   ${\rm Gr}_k(\CC^n)$ of all such subspaces. 
This map is a principal ${\rm U}(k)$-bundle, where the unitary group ${\rm U}(k)$ acts by matrix multiplication from the left on the Stiefel manifold.
Let us now identify any vector space $V$ in ${\rm Gr}_k(\CC^n)$ with the linear endomorphism $\pi_V$ of $\CC^n$ which represents the orthogonal projection of $\CC^n$  onto $V$ with respect to the inner product
mentioned above. The matrix of this endomorphism relative to the canonical basis is Hermitean and, if  
${\rm Herm}_n$ denotes the space of all such $n\times n$ complex matrices, one obtains 
\begin{equation}\label{gra}{\rm Gr}_k(\CC^n) = \{P \in {\rm Herm}_n \mid P^2=P, \ {\rm tr}(P)=k\}.\end{equation}
Like in Sect.~\ref{sect1}, let  $d=(d_1,\ldots, d_n) \in \RR^n$ be a vector with all $d_j \ge 0$. Then ${\mathcal F}_{n,k}^d$ is a subspace of
${\rm V}_k(\CC^n)$, clearly ${\rm U}(k)$-invariant. We wish to describe its image under the principal bundle map
mentioned above. To this end, we attach to any matrix $P \in {\rm Herm}_n$ the vector $\mu(P)$, which consists of the diagonal entries of $P$, and obtain in this way  a map $\mu : {\rm Gr}_k(\CC^n) \to \RR^n$. 
The principal bundle map ${\rm V}_k(\CC^n)\to {\rm Gr}_k(\CC^n)$ can be described as $F\mapsto F^* F$, for all $F\in {\rm V}_k(\CC^n)$, 
and consequently it transforms ${\mathcal F}_{n,k}^d$ into the pre-image  $\mu^{-1}(d)$. 
By also taking into account the topologies inherited from ${\rm Mat}_{k\times n}(\CC)$ and ${\rm Herm}_n$, one obtains a
homeomorphism
\begin{equation}\label{hom}{\mathcal F}_{n,k}^d /{\rm U}(k) \simeq \mu^{-1}(d).\end{equation}

\begin{rem}\label{sh} As a special case of the Schur-Horn theorem, see \cite{Schur} and \cite{Ho}, the image of $\mu: {\rm Gr}_k(\CC^n) \to \RR^n$ is just the polytope
$\Delta_{n,k}$ defined in Sect.~\ref{sect1} (see also \cite[Sect.~2.2]{Buch}). This explains why ${\mathcal F}_{n,k}^d$ is non-empty exactly if $d$ lies in this polytope.
\end{rem}

In order to prove Thm.~\ref{2mainthm}, we will now use the homeomorphism (\ref{hom}), combined with the knowledge of
the first two homotopy groups of  $\mu^{-1}(d)$. These will be computed in the next two sections. 

\section{Homotopy groups of $\mu^{-1}(d)$}

In this section we will make the first steps towards proving the following result.
\begin{thm}\label{mainthm}  If $d=(d_1, \ldots, d_n)\in \Delta_{n,k}$ satisfies assumption (\ref{hypothesis}), 
 then the first two homotopy groups of
 $\mu^{-1}(d)$ are: 
 \begin{align*}
 {}& \pi_1(\mu^{-1}(d))=\{0\}\\
 {}&   \pi_2(\mu^{-1}(d))\simeq \ZZ.
 \end{align*}
\end{thm}

Due to its rich structure, the Grassmannian mentioned in the previous section will be for us a useful tool. It admits a presentation as an  adjoint
orbit of the unitary group ${\rm U}(n)$ and is therefore canonically contained in the Lie algebra of this group, which is the space of all $n\times n$ skew-Hermitean matrices. By neglecting a factor of $i$, this viewpoint is exactly the one reflected by eq.~(\ref{gra}).
Let us now consider 
 the inner product 
\begin{equation} \label{metric}\langle X, Y\rangle ={\rm tr} (XY),\end{equation} $X, Y\in {\rm Herm}_n,$
  which is clearly invariant under the adjoint action of ${\rm U}(n)$. 
Like any adjoint orbit of a compact Lie group, ${\rm Gr}_k(\CC^n)$ can be equipped with a K\"ahler structure.
The space of all diagonal matrices in ${\rm U}(n)$ is an $n$-torus, call it $T$. The induced action of this torus on the Grassmannian is
Hamiltonian, and the corresponding moment map is just $\mu : {\rm Gr}_k(\CC^n) \to {\mathfrak t} $, which was defined above
(${\mathfrak t}$ denotes the Lie algebra of  $T$, which is naturally identified with ${\mathbb R}^n$).  
Strictly speaking, the moment map takes values in the dual space, ${\mathfrak t}^*$, but this is identified with 
${\mathfrak t}$ by means of the 
inner product (\ref{metric}). %Finally, note that the Riemannian metric coming from the aforementioned K\"ahler stucture on ${\rm Gr}_k(\CC^n)$ is exactly the one induced by eq.~(\ref{metric}) via restriction to the tangent spaces, see e.g.~\cite{EM}. 

We will investigate the pre-image $\mu^{-1}(d)$ by using results obtained  by Kirwan  in \cite{Ki}. 
The main tool is the ``norm squared of the moment map", that is the function  
 $f:{\rm Gr}_k(\CC^n) \to \RR$, $$f(V) = \|\mu(V)-d \|^2, \ V\in {\rm Gr}_k(\CC^n),$$
 where the norm is the one arising from the standard inner product on $\RR^n$. 
 More concretely, we will use  results from \cite[Sect.~4]{Ki}.
 One can see there that a prominent role is played by real-valued functions on the symplectic manifold which are obtained by pairing the moment map $\mu(\cdot) - d$ with elements  of ${\mathfrak t}$. In the case at hand, by the previous considerations, these are, modulo an additive constant, exactly the height functions which arise from the presentation (\ref{gra}) and the inner product (\ref{metric}). More precisely,   one is taking 
 $a=(a_1, \ldots, a_n) \in \RR^n$ and define $h_a: {\rm Gr}_k(\CC^n) \to \RR$, \begin{equation}\label{hadef}h_a(V)=\langle \pi_V,D_a\rangle,\end{equation}
where $D_a$ is the diagonal matrix ${\rm Diag} (a_1, \ldots, a_n)$. 
Now, these functions are well understood from the point of view of Morse theory. 
For example, this topic is approached nicely by Guest in \cite[Sect.~3]{Gu}. In what follows we will simply reproduce some of the  considerations from \cite{Gu}, with suitable additions  for which we refer to \cite[Sect.~2]{Mare}.

Let us first assume that 
\begin{equation}\label{ineqa}a_1\ge a_2 \ge \cdots \ge a_n.\end{equation}
The eigenspace decomposition of $D_a: \CC^n \to \CC^n$ is
$$\CC^n=E_1 \oplus \cdots \oplus E_\ell.$$
For any $u=\{u_1, \ldots , u_k\}\subset \{1, \ldots, n\}$  such that $u_1 < \cdots < u_k$ , set 
$$V_u:=\CC e_{u_1} \oplus \cdots \oplus \CC e_{u_k}$$
and consider its orbit 
$$M_{u}^a:={\rm U}(E_1) \times \cdots \times  {\rm U}(E_\ell) .V_u.$$
Note that 
$$V_u = (V_u \cap E_1) \oplus \cdots \oplus (V_u\cap E_\ell)$$
and consequently, by setting  $c_j:=\dim (V_u\cap E_j)$, $1\le j \le \ell,$ one has
\begin{equation}\label{mu}M_{u}^a={\rm Gr}_{c_1}(E_1) \times \cdots \times {\rm Gr}_{c_\ell}(E_\ell).\end{equation}
The importance of these submanifolds of ${\rm Gr}_k(\CC^n)$ resides in the fact that they are just the connected components 
of the critical set of $h_a$. Furthermore, if ${\rm Gr}_k(\CC^n)$ is equipped with the Riemann metric\footnote{This is actually the same as the canonical K\"ahler metric on the Grassmannian, viewed as an adjoint orbit of ${\rm U}(n)$, see e.g.~\cite{EM}.}  induced by the inner product (\ref{metric})
and the embedding described by eq.~(\ref{gra}),   the   stable manifold corresponding   to $M_{u}^a$
has the following presentation:
\begin{equation}\label{stablemnf}S_{u}^a = \{V \in {\rm Gr}_k(\CC^n) \mid \dim V \cap (E_1 \oplus \cdots \oplus E_j)= c_1+ \cdots + c_j, 
\ {\rm for \ all } \ 1\le j \le \ell\}.\end{equation}

Let us now drop the assumption (\ref{ineqa}). For any permutation $\tau \in S_n$ and any vector $x=(x_1, \ldots, x_n)$, we denote  
$$x^\tau:=(x_{\tau(1)}, \ldots, x_{\tau(n)}).$$ 
Getting back to $a=(a_1, \ldots, a_n)$, which is arbitrary in $\RR^n$, there exists $\sigma\in S_n$ such that
 \begin{equation}\label{orfera}a_{\sigma(1)} \ge \cdots \ge a_{\sigma(n)}.\end{equation}
 Let also $g$ be the linear transformation of $\CC^n$ such that
 $$ge_j=e_{\sigma(j)},$$
 for all $1\le j \le n$, where $e_1, \ldots, e_n$ is the canonical basis of $\CC^n$. 
 
 \begin{prop}\label{critgen}
\begin{itemize}
\item[(i)] For any $V\in {\rm Gr}_k(\CC^n)$, one has 
\begin{equation} \label{ha} h_{{a}^\sigma}(V)=h_a(gV)\end{equation}
and
\begin{equation} \label{mug} \mu(gV)=\mu(V)^{\sigma^{-1}}.\end{equation}
\item[(ii)] 
The function $h_a$ is Morse-Bott and its critical set  is  
\begin{equation}\label{ch}{\rm Crit}(h_a)=g{\rm Crit}(h_{{a}^\sigma})=\bigcup gM_{u}^{{a}^\sigma},\end{equation}
where ${u\subset \{1,\ldots, n\}}$ with $k$ elements.  
Furthermore, the  stable manifold corresponding to $gM_{u}^{{a}^\sigma}$ is $gS_{u}^{{a}^\sigma}$, where
$S_{u}^{{a}^\sigma}$ is described by eq.~(\ref{stablemnf}). 
\end{itemize}
\end{prop}

 \begin{rem}\label{complex} We note that ${\rm Gr}_k(\CC^n)$ is canonically a complex manifold and  $S_{u}^a$  a complex submanifold of it, see e.g.~\cite[p.~170]{Gu}. %\marginpar{Schubert cell ?}
\end{rem}

\begin{rem}\label{ambig2} 
For a given $a\in \RR^n$ there are in general several ways to choose $\sigma\in S_n$ leading to the same $a^\sigma$.
After choosing such a $\sigma$, there are several $u\subset \{1, \ldots, n\}$ such that $gM_{u}^{{a}^\sigma}$ is the same.  
\end{rem}

\begin{rem} \label{unif} Assume $a=(x, \ldots, x)$, that is, all components of $a$ are equal to each other. Then 
$h_a(V) = \langle \pi_V, x{\rm I}_n\rangle = {\rm tr}(x\pi_V)=xk$, which means that the function $h_a$ is constant. 
Thus all $V\in {\rm Gr}_k(\CC^n)$ are critical points for $h_a$. This is also confirmed by the theory above, since 
we clearly have $\ell =1$ and for any $u\subset \{1, \ldots, n\}$ with $k$ elements, $M_u^a$ is the whole 
${\rm Gr}_k(\CC^n)$. 
\end{rem}

 All this  enables us to state:
  
\begin{thm}\label{stratif} \begin{itemize} 
\item[(i)] The critical set of $f_{}$ is the union of all (non-empty) intersections of the form
$\mu^{-1}(d+a) \cap gM_{u}^{{a}^\sigma}$ where $a\in \RR^n$  and $u\subset \{1, \ldots, n\}$ with $k$ elements.
The intersection is non-empty  only for finitely many  $a$ and $u$ as above. For any such $a$ and $u$, there is a submanifold
$\Sigma_{a, u}$  of ${\rm Gr}_k(\CC^n)$ which contains $\mu^{-1}(d+a) \cap gM_{u}^{{a}^\sigma}$
as a deformation retract. These submanifolds induce a partition 
\begin{equation}\label{stratsigma}{\rm Gr}_k(\CC^n)= \bigsqcup \Sigma_{a, u}.\end{equation}
The codimension of $\Sigma_{a,u}$ in
${\rm Gr}_k(\CC^n)$ is equal to the index of $h_a$ along $gM_{u}^{{a}^\sigma}$, that is, the codimension of
$gS_{u}^{a^\sigma}$.
\item[(ii)] The partition (\ref{stratsigma}) is a stratification in the sense of \cite[Def.~2.11]{Ki}. More precisely, one can label the submanifolds as $\{\Sigma_\beta  \mid \beta\in B\}$, where $B$ is a finite set endowed with a strict partial order ``$ \ >$" such that for any $\beta \in B$ the topological closure of $\Sigma_\beta$ satisfies
\begin{equation}\label{osig} \overline{\Sigma}_\beta \subseteq \bigcup_{\gamma \ge \beta} \Sigma_\gamma.\end{equation} 
\item[(iii)] There is exactly one stratum $\Sigma_{a,u}$ of codimension zero, namely 
the one corresponding to $a=(0, \ldots, 0)$ 
and $u$  any subset with $k$ elements of\footnote{See also Rem.~\ref{unif}. Also note the general fact that if $a=(a_1, \ldots, a_n)$ such that $\mu^{-1}(d+a)$ is non-empty, then $a_1+\cdots + a_n=0$, because the sum of the components of both $d$ and $d+a$ is equal to $k$.}
$\{1, \ldots, n\}$.  Moreover,  $M_{u}^{(0, \ldots, 0)}$ is the whole ${\rm Gr}_k(\CC^n)$ and hence 
$\Sigma_{(0, \ldots, 0),u}$ contains $\mu^{-1}(d)$ as a deformation retract. In terms of the notations from point (ii), this stratum is of the form $\Sigma_{\beta_0}$, where $\beta_0\in B$ is the greatest element.
\end{itemize} 
\end{thm}

\begin{proof}
(i) According to \cite[Thm.~4.16]{Ki}, the set ${\rm Crit}(f)$ is labeled by $a\in \RR^n$ and $u\subset \{1, \ldots, n\}$ with $k$ elements
such that \begin{equation}\label{gma}\langle \mu(gM^{a^\sigma}_u)-d, a\rangle =\|a\|^2.\end{equation}
Furthermore, for any such $a$ and $u$, the contribution to ${\rm Crit}(f)$ is
$\mu^{-1}(d+a) \cap gM^{a^\sigma}_u$. But if the latter intersection is non-empty, then eq.~(\ref{gma}) is automatically satisfied:
$d+a\in \mu(g M_u^{a^\sigma})$ implies that the left hand side of eq.~(\ref{gma}) is equal to $h_a(gM_u^{a^\sigma})-\langle d,a \rangle$, which is constant, since $gM_u^{a^\sigma}$ is a component of the critical set of $h_a$, the latter being a Morse-Bott function, see Prop.~\ref{critgen}; the aforementioned constant must then be equal to $\langle (d+a)-d, a\rangle =\|a\|^2$. 

The statement concerning the codimension of $\Sigma_{a, u}$ follows from \cite[Rem.~4.17]{Ki}. The fact that this manifold contains $\mu^{-1}(d+a) \cap gM_{u}^{{a}^\sigma}$
as a deformation retract is a direct application of a general result of Duistermaat, see \cite{Le}. 

(ii) This is also a direct consequence of \cite[Thm.~4.16]{Ki} (see also \cite[Lemma 10.7]{Ki}).  

(iii) The stratum involved in this statement is the one corresponding to the minimum level of $f$, which is exactly
$\mu^{-1}(d)$. 

\end{proof}

The information provided by this theorem can now be used to accomplish our actual goals, as we will see in the next section.

\section{Proof of Theorems \ref{mainthm} and \ref{2mainthm}}

We need some preparation before being able to actually start proving the theorems.

\begin{prop} \label{prop1ton}
Assume $d\in \Delta_{n,k}$ satisfies  assumption (\ref{hypothesis}). If $a\in \RR^n$ and $u\subset \{1, \ldots, n\}$ with $k$ elements are such that 
$\mu^{-1}(d+a) \cap g M_{u}^{{a}^\sigma}$ 
is non-empty and the codimension of $\Sigma_{a,u}$ is strictly greater than 0
then the codimension of $\Sigma_{a,u}$  is at least equal to 4.
\end{prop}

 \begin{proof} %In this proof when we refer to dimension or codimension it will be the complex one. 
 By Thm.~\ref{stratif} (i), $\Sigma_{a,u}$ and $S_{u}^{{a}^\sigma}$ have the same codimension.
 By Rem.~\ref{complex}, it will be sufficient to show that the {\it complex} codimension of the latter manifold is not equal to 1. 
 This can be proved by contradiction, using the same method as in \cite{Mare}, proof of Prop.~4.1, with some obvious changes.  The only essential modification is that this time, if $E\subset \CC^n$ is an $m$-dimensional vector subspace
 and $c$ an integer, $1 \le c \le {\rm min} \{m, k\}$, then
 the space $\{V \in {\rm Gr}_k(\CC^n) \mid \dim V \cap E =c\}$ is a complex submanifold of $ {\rm Gr}_k(\CC^n)$ of complex codimension equal to  $c(n-k-m+c)$, see \cite[Thm.~4.1 (c)]{Wo}. 
  \end{proof}

 Recall now that, by definition,  a continuous map $f: (X, x_0)\to (Y, y_0)$ between two pointed topological spaces is {\it m-connected} if
$f_*: \pi_k(X, x_0) \to \pi_k(Y, y_0)$ is an isomorphism for $0\le k \le m-1$ and is surjective for $k=m$. The following general result is the content of \cite[Thm.~1.11]{BW}:

\begin{lem}  \label{cs} {\rm (B.~Williams)} Let $M$ be a connected manifold and $Z\subset M$ a closed submanifold of codimension $d\ge 2$. Then the inclusion map 
$M\setminus Z \hookrightarrow M$ is $(d-1)$-connected.  
\end{lem}

Note that $Z$  can be a finite union of connected closed submanifolds, $d$ being the minimal codimension of a component.  

This will be used in proving the following proposition. 

\begin{prop} \label{key} If $d$ satisfies assumption (\ref{hypothesis}), then the inclusion map
$\mu^{-1}(d) \hookrightarrow  {\rm Gr}_k(\CC^n)$ is 3-connected.
\end{prop} 

\begin{proof}
The map mentioned in the proposition can be decomposed as 
$$\mu^{-1}(d) \hookrightarrow \Sigma_{(0,\ldots, 0),u} \hookrightarrow {\rm Gr}_k(\CC^n),$$
where $u\subset \{1, \ldots n\}$ is an arbitrary subset with $k$ elements. 
By Thm.~\ref{stratif} (iii) above, the first of these maps is a deformation retract and thus a homotopy equivalence.

We will now analyze the other inclusion map, which is $\Sigma_{\beta_0}\hookrightarrow  {\rm Gr}_k(\CC^n)$, see again
Thm.~\ref{stratif} (iii). We clearly have
\begin{equation}\label{zero}\Sigma_{\beta_0} = {\rm Gr}_k(\CC^n)\setminus \bigcup_{\beta \neq \beta_0} \Sigma_\beta.\end{equation}
Let $\beta^1_1, \ldots, \beta^1_{m_1}$ be the minimal elements of the indexing set $B$. By eq.~(\ref{osig}), the corresponding
$\Sigma$'s are closed in ${\rm Gr}_k(\CC^n)$. Prop.~\ref{prop1ton} and Lemma \ref{cs} imply that the map
$${\rm Gr}_k(\CC^n) \setminus \bigcup_{i=1}^{m_1}\Sigma_{\beta_i^1}\hookrightarrow {\rm Gr}_k(\CC^n) $$
is 3-connected. Let now $\beta^2_1, \ldots, \beta^2_{m_2}$ be the minimal elements of 
$B\setminus \{\beta^1_1, \ldots, \beta^1_{m_1}\}$. Again by eq.~(\ref{osig}), the spaces $\Sigma_{\beta_1^2}, \ldots, 
\Sigma_{\beta_{m_2}^2}$ are all closed in ${\rm Gr}_k(\CC^n) \setminus \bigcup_{i=1}^{m_1}\Sigma_{\beta_i^1}$. 
And by applying again Prop.~\ref{prop1ton} and Lemma \ref{cs}, the map
 $${\rm Gr}_k(\CC^n) \setminus \left(\bigcup_{i=1}^{m_1}\Sigma_{\beta_i^1}\cup \bigcup_{j=1}^{m_2}\Sigma_{\beta_j^2}\right)\hookrightarrow {\rm Gr}_k(\CC^n) \setminus \bigcup_{i=1}^{m_1}\Sigma_{\beta_i^1}$$
 is $3$-connected. The procedure can be continued and leads to the following chain of strict inclusions:
 $$B\supset B\setminus \{\beta_1^1, \ldots, \beta_{m_1}^1\}\supset B\setminus \{\beta_1^1, \ldots, \beta_{m_1}^1,
 \beta_1^2,\ldots, \beta_{m_2}^2\} \supset \ldots .$$
 The last set in the chain must be $\{\beta_0\}$, since $\beta_0$ is the greatest element of $B$ and hence not a minimal element of any subset of $B$ which contains it properly. 
We conclude that, in view of eq.~(\ref{zero}) above, the map
 $$\Sigma_{\beta_0} \hookrightarrow  {\rm Gr}_k(\CC^n) $$
 is a composition of several 3-connected maps and thus is itself 3-connected. 
\end{proof}

We are now ready to achieve the two main goals of this section:

\vspace{0.2cm}

\noindent{\it Proof of Theorems \ref{mainthm} and \ref{2mainthm}.} 
Thm.~\ref{mainthm} follows immediately from Prop.~\ref{key} and the well-known
facts that 
\begin{align*}
{}& \pi_1({\rm Gr}_k(\CC^n) )=\{0\},
\\
& \pi_2({\rm Gr}_k(\CC^n) )\simeq \ZZ,
\end{align*}
see for instance \cite[p.~300]{BT}.

To prove the second theorem, we return to the considerations made at the end of Sect.~\ref{Grass}. 
We mentioned the ${\rm U}(k)$-principal bundles ${\rm V}_k(\CC^n) \to {\rm Gr}_k (\CC^n)$ and ${\mathcal F}_{n,k}^d \to \mu^{-1}(d)$. Fix a point ${\rm pt.}\in \mu^{-1}(d)$ and consider its fiber, which is homeomorphic to ${\rm U}(k)$. The inclusion map
$({\mathcal F}_{n,k}^d, {\rm U}(k)) \hookrightarrow ({\rm V}_k(\CC^n) , {\rm U}(k))$ gives rise to a homomorphism between homotopy groups. Combined with the long exact homotopy sequences of the two pairs of spaces, we are led to the following 
``ladder" diagram:

\begin{equation}\label{exacts}
\vcenter{\xymatrix{
&
\pi_2({\mathcal F}_{n,k}^d, {\rm U}(k)) 
\ar[r]^{ \ \ \textcircled{\small{1}}}\ar[d]^{\textcircled{\small{2}}} &
\pi_1({\rm U}(k)) 
\ar[r] \ar[d]^{\simeq} &
\pi_1({\mathcal F}_{n,k}^d)\ar[r] \ar[d]&  \pi_1({\mathcal F}_{n,k}^d, {\rm U}(k))\ar[d]
\\
&
\pi_2({\rm V}_k(\CC^n), {\rm U}(k))
\ar[r]^{ \ \ \ \ \ \textcircled{\small{3}}} &
\pi_1({\rm U}(k))  \ar[r] & \pi_1({\rm V}_k(\CC^n)) \ar[r] &  \pi_1({\rm V}_k(\CC^n), {\rm U}(k))
\\
}}
\end{equation}  
Here the two horizontal lines are exact sequences and the rectangles are commutative diagrams.
This follows as an application of a general result which can be found for instance in \cite[p.~76]{Gray}. 

{\it Claim.} The map $\textcircled{\small{2}}$ is an isomorphism.

To prove this, we consider the following commutative diagram

\begin{equation*}\label{exactseq}
\vcenter{\xymatrix{
&
({\mathcal F}_{n,k}^d, {\rm U}(k)) 
\ar[r]^{}\ar[d]^{} &
(\mu^{-1}(d), {\rm pt.}) 
\ar[d]^{} 
\\
&
({\rm V}_k(\CC^n), {\rm U}(k))
\ar[r]^{} &
({\rm Gr}_k(\CC^n), {\rm pt.})  
\\
}}\
\end{equation*}  

\noindent where the horizontal arrows indicate the bundle maps (see also eq.~(\ref{hom})), and the vertical arrows are the inclusions.
By functoriality, one obtains an obvious commutative diagram which involves the second homotopy groups. In this new diagram, 
the horizontal arrows are isomorphisms, see \cite[Thm.~11.8]{Gray}. The left hand side vertical arrow is just the map 
$\textcircled{\small{2}}$ above and the other vertical arrow is an isomorphism, according to Prop.~\ref{key} above. This concludes the proof of the claim.

Recall that $\pi_1( {\rm V}_k(\CC^n))=\{0\}$, see for instance \cite[Thm.~3.15, p.~68]{MM} and note that, by Rem.~\ref{lege}, we have $2 \le k \le n-2$. 
Consequently, the map ${\textcircled{\small{3}}}$ is surjective. Since   ${\textcircled{\small{2}}}$ is an isomorphism, 
${\textcircled{\small{1}}}$ is surjective as well. 
From \cite[Thm.~11.8]{Gray} and Thm.~\ref{mainthm} it follows that 
$$\pi_1({\mathcal F}_{n,k}^d, {\rm U}(k)) \simeq \pi_1(\mu^{-1}(d))\simeq \{0\}.$$ The exactness of the sequence on the top line 
implies that $\pi_1({\mathcal F}_{n,k}^d)=\{0\}$, as desired.
$\square$

\begin{rem}  \label{meth} As noticed in the introduction, there are similarities between this work and \cite{CNS}, also with regard to the methods used in proving the main results. More precisely, for both \cite[Thm.~4.5]{CNS} and Thm.~3.1 above, homotopy groups of a  subspace of a well-understood manifold   are computed by means of a stratification of the manifold, to which belongs a unique open stratum, where the latter contains the subspace one is interested in as a deformation retract, the crucial ingredient being that all the non-open strata have sufficiently high codimensions. This general idea goes back to Kirwan \cite{Ki}, and was already used in the context of frames in \cite{Mare}. 
\end{rem}

\section{Examples}

\begin{ex}\label{exa1} The case when $n=2$ and $k=1$ can be easily understood.    First, ${\rm Gr}_1(\CC^2)=\CC P^1$ is a 2-sphere and 
$\mu$ can be viewed as the orthogonal projection of the sphere on one of its diameters. Its levels are almost all  non-degenerate circles and hence not simply connected. The Stiefel manifold 
${\rm V}_1(\CC^2)$ is just the unit sphere
$S^3$ in $\CC^2$, and
$$\Delta_{2,1}=\{(d_1, d_2) \in \RR^2 \mid d_1, d_2 \ge 0, d_1+d_2=1\}.$$ Consequently, ${\mathcal F}_{2,1}^d$ are products of two circles and consequently not simply connected.
\end{ex}

%We will be looking at ${\rm Gr}_2(\RR^4)$ with two choices of $d$ in the corresponding polytope $\PP_{4,2}$. 

\begin{ex} \label{exa} For $n=4$ and $k=2$ we consider $d=(1,1,0,0)$. Let $\langle \ ,  \ \rangle$ be, as before, the standard Hermitean inner product on $\CC^4$.  Then $\mu^{-1}(d)$ consists of all complex
2-planes $V$ in $\CC^4$ such that
\begin{align*}
{}& \langle \pi_V(e_1), e_1 \rangle =1\\
 {}& \langle \pi_V(e_2), e_2 \rangle =1\\
 {}& \langle \pi_V(e_3), e_3 \rangle =0\\
 {}& \langle \pi_V(e_4), e_4 \rangle =0.
 \end{align*}
The first two conditions show that $e_1$ and $e_2$ are in $V$, in other words,  that $V$ is the span of $e_1$ and $e_2$.
So in this case $\mu^{-1}(d)$ is just a point in ${\rm Gr}_2(\CC^4)$. 
  Since ${\mathcal F}_{4, 2}^{(1,1,0,0)}$ is acted freely by ${\rm U}(2)$, the orbit being just a point, it must be homeomorphic to
  ${\rm U}(2)$. 
  %are matrices of the type $[ \ A \ | \ 0 \ ]$ where both $A$ and $0$ are $2\times 2$ matrices, the latter having all entries equal to 0 and the former being unitary. 
  Thus $ {\mathcal F}_{4, 2}^{(1,1,0,0)}$ is not simply connected. As expected,  $d=(1,1,0,0)$ is  not satisfying assumption (\ref{hypothesis}).\end{ex}

\begin{ex} Again for $n=4$ and $k=2$, we now take $d=(\frac{1}{3},\frac{1}{3},\frac{1}{3},1)$. This time the 2-planes $V$ in $\mu^{-1}(d)$ are determined by:
\begin{align*}
{}& \langle \pi_V(e_1), e_1 \rangle =\frac{1}{3}\\
 {}& \langle \pi_V(e_2), e_2 \rangle =\frac{1}{3}\\
 {}& \langle \pi_V(e_3), e_3 \rangle =\frac{1}{3}\\
 {}& \langle \pi_V(e_4), e_4 \rangle =1.
 \end{align*}
 From the last condition, $e_4\in V$. So $V$ is uniquely determined by its quotient $V/ \CC e_4$,
 which is a line, say $\ell$, in $\CC^4/ \CC e_4 \simeq \CC^3$ such that
 \begin{align*}
 {}& \langle \pi_\ell(e_1), e_1 \rangle =\frac{1}{3}\\
 {}& \langle \pi_\ell(e_2), e_2 \rangle =\frac{1}{3}\\
 {}& \langle \pi_\ell(e_3), e_3\rangle =\frac{1}{3}.
 \end{align*}
 Now $\ell$ is in $\CC P^2$ and we write it as $\ell = \CC v$, where $v=(a,b,c)\in \CC^3$, $|a|^2+|b|^2+|c|^2=1$. 
 Then $\pi_\ell (e_1)=\overline{a}v$, $\pi_\ell(e_2)=\overline{b}v$, and $\pi_\ell (e_3)=\overline{c}v$,
 thus the three equations above are equivalent to
 $$|a|=|b|=|c|=\frac{1}{\sqrt{3}}.$$ In conclusion, $\mu^{-1}(d)$ is the quotient of $S^1\times S^1\times S^1$ by the diagonal action of $S^1$, thus a 2-torus, which is not simply connected.  
 Let us now consider the following piece of the long exact homotopy sequence of the ${\rm U}(2)$-principal bundle
 ${\mathcal F}_{4,2}^{(\frac{1}{3}, \frac{1}{3}, \frac{1}{3},1)} \to \mu^{-1}(d)$: 
 $$  \pi_1({\mathcal F}_{4,2}^{(\frac{1}{3}, \frac{1}{3}, \frac{1}{3},1)}) \longrightarrow \pi_1(\mu^{-1}(d))\longrightarrow \pi_0({\rm U}(2))=\{0\}.$$
 Since $ \pi_1(\mu^{-1}(d)) \neq \{0\}$, it follows that  $\pi_1({\mathcal F}_{4,2}^{(\frac{1}{3}, \frac{1}{3}, \frac{1}{3},1)})$ is non-zero as well.  
\end{ex}

\begin{ex} \label{exa4}
At a more general level, examples of frames which satisfy the hypothesis of Thm.~\ref{2mainthm}  are provided by \cite[Sect.~4]{CNS}. 
To describe them, let us assume that $k\ge 2$ and consider $d=(d_1, \ldots, d_n) \in \Delta_{n,k}$. 
There exists a permutation $\sigma \in S_n$ such that $$d_{\sigma(1)} \ge d_{\sigma(2)} \ge  \cdots\ge d_{\sigma(n)}.$$ 
Thm.~4.5 in \cite{CNS} gives  sufficient conditions for ${\mathcal F}_{n,k}^d$ to be $q$-connected, where $q\ge 0$ is an arbitrary integer. For $q=1$  these assumptions amount to:
\begin{align*} {}& \bullet \ d_1, \ldots, d_n \in {\mathbb Q},\\
{}& \bullet \ \frac{k}{n\ell} \cdot {\rm min} \left\{ j \mid \sum_{i=1}^jd_{\sigma(i)}>\ell\right\}
\ge \frac{k}{2n}\cdot \frac{2k+1}{k-1}, \ {\it for \ all \ }
\ell \in \{1, \ldots, k-1\}.\end{align*}
In the special case when $\ell = k-1$ one obtains
$${\rm min} \left\{ j \mid \sum_{i=1}^jd_{\sigma(i)}>k-1 \right\}  \ge k+1. $$
This implies 
$$\sum_{i=1}^kd_{\sigma(i)}\le k-1,$$
which, in turn, is equivalent to
$$k-\sum_{i=k+1}^nd_{\sigma(i)} \le k-1,$$
and further to
$$\sum_{i=k+1}^nd_{\sigma(i)}\ge 1.$$
This is exactly assumption (\ref{hypothesis}) in Thm.~\ref{2mainthm}.
\end{ex}

\noindent {\bf Acknowledgement.} I wish to thank the referee for carefully reading the manuscript and suggesting substantial improvements.

\end{document}